\documentclass[reqno]{amsart}
\makeatletter
\@namedef{subjclassname@2020}{%
  \textup{2020} Mathematics Subject Classification}
\makeatother
\usepackage{amssymb}
\usepackage[utf8]{inputenc} % allow utf-8 input
\usepackage[T1]{fontenc}    % use 8-bit T1 fonts
\usepackage{url}            % simple URL typesetting
\usepackage{booktabs}       % professional-quality tables
\usepackage{amsfonts}       % blackboard math symbols
\usepackage{nicefrac}       % compact symbols for 1/2, etc.
\usepackage{microtype}      % microtypography
\usepackage{color}
\usepackage{lipsum}
\usepackage{float}
\usepackage{calrsfs}
\usepackage{enumerate}
\usepackage{tikz-cd}
\usepackage{tikz}

\usetikzlibrary{positioning}
\usepackage{graphicx}
\usepackage{old-arrows} 
\usetikzlibrary{arrows}
\usepackage{comment}
\DeclareMathAlphabet{\pazocal}{OMS}{zplm}{m}{n}

\usepackage{eqnarray,amsmath}

\input xygraph
\input xy
\xyoption{all}
\def\reg{\mathop{\text{reg}}}

\def\F{\mathcal F}

\def\Q{\mathbb Q}

\def\a{\alpha}
\def\b{\beta}
\def\l{\lambda}

\def\w{\omega}

\def\remove#1{}

\def\M{\mathop{\mathcal M}}
\DeclareMathOperator{\sop}{supp}

\newcommand{\C}{\mathcal{C}}
\newcommand{\ann}{{\rm Ann}}
\DeclareMathOperator{\Span}{span}

\def\rad{\mathcal{R}_{abs}}
\def\asi{\rm{asi}}

\def\top{\bf DCC}

\def\P{\mathcal{P}}
\def\azd{\mathcal{D}}
\def\vacio{\emptyset}
\newtheorem{lemma}{Lemma}[section]
\newtheorem{corollary}[lemma]{Corollary}
\newtheorem{theorem}[lemma]{Theorem}
\newtheorem{proposition}[lemma]{Proposition}
\newtheorem{remark}[lemma]{Remark}

\newtheorem{example}[lemma]{Example}

\title[Properties of evolution algebras through graphs]{Centroid and algebraic properties of evolution algebras through graphs}
\author[Y. Cabrera]{Yolanda Cabrera Casado}
\address{Y. Cabrera Casado: Departamento de Matem\'atica Aplicada, E.T.S. Ingenier\'\i a Inform\'atica, Universidad de M\'alaga, Campus de Teatinos s/n. 29071 M\'alaga.   Spain.}
\email{yolandacc@uma.es}

\author[M. I. Gon\c calves]{Maria Inez Cardoso Gon\c calves}
\address{M. I. Cardoso Gon\c calves: Departamento de Matem\'atica, Universidade Federal de Santa Catarina, Florian\'opolis, SC, 88040-900 - Brazil}
\email{maria.inez@ufsc.br}

\author[D. Gon\c calves]{Daniel Gon\c calves}
\address{D. Gon\c calves: Departamento de Matem\'atica, Universidade Federal de Santa Catarina, Florian\'opolis, SC, 88040-900 - Brazil}
\email{daemig@gmail.com}

\author[D. Mart\'\i n]{Dolores Mart\'\i n Barquero}
\address{D. Mart\'\i n Barquero: Departamento de Matem\'atica Aplicada, Escuela de Ingenier\'\i as Industriales, Universidad de M\'alaga, Campus de Teatinos s/n. 29071 M\'alaga.   Spain.}
\email{dmartin@uma.es}

\author[C. Mart\'\i n]{C\'andido Mart\'\i n Gonz\'alez}
\address{C. Mart\'\i n Gonz\'alez: Departamento de \'Algebra Geometr\'{\i}a y Topolog\'{\i}a, Fa\-cultad de Ciencias, Universidad de M\'alaga, Campus de Teatinos s/n. 29071 M\'alaga.   Spain.}
\email{candido\_m@uma.es}

\author[I. Ruiz]{Iv\'an Ruiz Campos}
\address{I. Ruiz Campos:  Departamento de \'Algebra Geometr\'{\i}a y Topolog\'{\i}a, Fa\-cultad de Ciencias, Universidad de M\'alaga, Campus de Teatinos s/n. 29071 M\'alaga. Spain.}
\email{ivaruicam@uma.es}


\begin{document}
\maketitle
\begin{abstract}
    The leitmotiv of this paper is linking algebraic properties of an evolution algebra with combinatorial properties of the (possibly several) graphs that one can associate to the algebra. We link  nondegeneracy, zero annihilator,  absorption property, von Neumann regularity and primeness with suitable properties in the associated graph. In the presence of semiprimeness, the property of primeness is equivalent to any associated graph being downward directed. We also  provide a description of the prime ideals in an evolution algebra and prove that certain algebraic properties, such as semiprimeness and perfection, can not be characterized in combinatorial terms.
    We describe the centroid of evolution algebras as constant functions along the connected components of  its associated graph. The dimension of the centroid of a zero annihilator algebra $A$ agrees with the cardinal of the connected components of any possible graph associated to $A$. This is the combinatorial expression of an algebraic uniqueness property in the decomposition of $A$ as indecomposable algebras with $1$-dimensional centroid. 
\end{abstract}
\section{Introduction}

The term "algebraic combinatorics" was coined in the early 1970s by Bannai. The subject has gained popularity with the publication of the book  \cite{Bannai} and has since emerged as a prominent example of interactions between areas of Mathematics. Nowadays, the term "combinatorial algebra" loosely refers to any algebra constructed from a combinatorial object; see \cite{exel}. Recent examples include graph, ultragraph, and higher-rank graph algebras, all of which have strong ties with combinatorics and symbolic dynamics; see \cite{williams} for example.  Across these examples, much of their algebraic structure can be linked to the combinatorial properties of the underlying objects, facilitating their comprehension, analysis, and application in other areas.

Evolution algebras were initially defined as models for the study of non-Mendelian genetics in \cite{tianvoj}, and their exploration has significantly expanded in recent years. In \cite{Elduque2}, a combinatorial structure was associated to an evolution algebra: a directed graph. However, this graph's uniqueness is contingent upon the chosen algebraic basis, presenting challenges for applying combinatorial techniques to evolution algebras. In \cite{squares}, the authors propose to relate all possible graphs associated to an evolution algebra. While this approach has been successful for two-dimensional algebras, extending it to higher dimensions remains an open challenge.

Despite the difficulties pointed out above, there have been some advancements in connecting evolution algebras to their associated graphs. For instance, in \cite{Elduque}, the group of automorphisms is related to a graph of an evolution algebra, and in \cite{conecting}, ideals in evolution algebras are associated to hereditary subsets of its graph. Strikingly, the combinatorial characterization of some algebraic properties of evolution algebras (such as the link between ideals and hereditary subsets (\cite{conecting}), and the link between the graph downward directed condition and primeness of the algebra (Corollaries~\ref{cor_prime_eq_down} and \ref{cool})) mirror those in Leavitt path algebras, despite the former being commutative and non-associative whereas the latter is non-commutative and associative. Furthermore, the centroid of simple evolution algebras and simple Leavitt path algebras agree (being equal to the base field, see \cite{centroid} and Theorem~\ref{praiajurere}). These findings suggest a connection between these two classes of algebras, the formal nature of which (if any) remains a mystery.

In this article, we aim to establish new connections between the structure of evolution algebras and the combinatorial properties of their associated graphs, contributing to the inclusion of evolution algebras in the class of "combinatorial algebras". Moreover, we also identify instances when an algebraic property of an evolution algebra can not be inferred from their associated graphs, such as semiprimeness and perfection.

After a preliminary session, where we also show that the (combinatorial) connected components of a graph agree with the (topological) components of the DCC topology (introduced in  \cite{CMMSS}) on the graph, in the following sections, we study the following properties of an evolution algebra.

In Section~3, we show that nondegeneracy and having zero annihilator can be seen from any graph associated to an evolution algebra, see  Proposition~\ref{ostra}. In particular, an evolution algebra has zero annihilator if and only if the graph is sinkless. In this proposition, we also link the existence of sinks with the condition of degeneracy. 
Next, in Proposition~\ref{nocaracsemiprime}, we show that the semiprimeness and perfection of an evolution algebra do not have a graphical characterization.  This is exemplified by presenting two algebras with the same graph,  where one is semiprime, and the other is not. Given this shortcoming of graphs, we provide an algorithm, using the structure matrix of an evolution algebra, to decide whether the algebra is semiprime, see Section~\ref{jabutiaqui} (in this section, we also provide an algorithm to decide the nondegeneracy of the algebra). In the case of perfect algebras, we completely characterize the existence of nonzero absolute zero divisors in terms of sinks in a subgraph, see Proposition~\ref{colorida}.  We show that the downward directed condition is equivalent to primeness in Corollary~\ref{cool} (in fact, this result holds not only under the assumption of perfection but also under semiprimeness, see Corollary~\ref{cor_prime_eq_down}).  We continue the investigation of primeness with a description of the prime ideals in an evolution algebra, see Proposition~\ref{rp3barulhento} and Theorem~\ref{charprimo}. The absorption radical of an evolution algebra is described in graphical terms in Proposition~\ref{paella}. We finish Section~3 with a study of von Neumann regularity, and show that an evolution algebra is von Neumann regular if and only if every associated graph consists of isolated loops.

The centroid of an evolution algebra is the focus of Section~4. We show that the number of connected components in a graph associated to an evolution algebra can be read from the centroid. More precisely, for zero annihilator algebras, we show that the centroid of the algebra is isomorphic to the product, over the set of connected components of any graph associated to the algebra, of the base field, see Theorems~\ref{praiajurere} and \ref{praiaMole}. Consequently, since we have shown that prime algebras have only one component (Theorem~\ref{despierto}), we get that the centroid of a prime evolution algebra is isomorphic to the base field. In our analysis, we make use of a process of simultaneously diagonalization of a family of linear operators, a procedure that is interesting on its own (see the beginning of Section~4).

\section{Preliminaries }

Consider  a ring $S$ and an element $a \in S$. We define the \emph{left multiplication operator} $L_a \colon S \to S$ by $L_a(b) := ab$. Analogously, we define the \emph{right multiplication operator} $R_b$ for a $b \in S$. The condition for $S$ to be commutative is equivalent to $L_a = R_a$ for every $a \in S$.

An algebra $A$ over a field $K$ is considered an \emph{evolution algebra} if there exists a basis $B=\{e_i\}_{i\in \Lambda}$ such that $e_ie_j=0$ for every $i, j \in \Lambda$ with $i\neq j$. Such a basis is called a \emph{natural basis}. Denote by $M_B=(\omega_{ij})$ the \emph{structure matrix} of $A$ relative to $B$, where $e_i^2 = \sum_{j\in \Lambda} \omega_{ji}e_j$. For $x=\sum_{i \in \Lambda} \lambda_{i} e_i$ , the support of $x$ with respect to $B$ is denoted as $\text{supp}_B(x)=\{i \in \Lambda:\lambda_i\neq 0\}$. We will denote the basis just by $\sop(x)$ when the basis is clear. For a set $X$,  we define $\sop_B(X) = \cup_{x \in X} \sop_B(x)$.

We include a summary of definitions and results that we will apply to evolution algebras along this work:
\begin{enumerate}
    \item {\it Semiprime}: an algebra $A$ such that for any $I\triangleleft A$ with $I^2=0$ one has $I=0$.
   
    \item {\it Nondegenerate}: an algebra $A$ such that $(xA)x=0$ implies $x=0$. The pathological elements $x\in A$ verifying $(xA)x=0$ are called {\it absolute zero-divisors}.  The set of absolute zero divisors of $A$ will be denoted by $\azd(A)$.   These elements were introduced for alternative algebras (though without using this name) in \cite{Kleinfeld}. 
    Since then, they have been intensively used in the theory of Jordan algebra and others.  An algebra containing a nonzero absolute zero divisor is called {\it degenerate.} It is easy to see that if $A$ is nondegenerate, then it is semiprime.
    \item {\it Decomposable}: an algebra $A$ which splits in the form $A=I\oplus J$, where $I,J\triangleleft A$, $I,J\ne 0$.
    \item {\it Prime}: an algebra $A$ such that for any $I,J\triangleleft A$ with $IJ=0$ one has $I=0$ or $J=0$.  A proper ideal $I$ of $A$ is said a {\it prime ideal} if the quotient ring $A/I$ is prime.
    \item {\it Perfect}: an algebra verifying $A=A^2$.  For finite-dimensional evolution algebras, this is equivalent to $\vert M_B\vert \neq 0$. 
    \item {\it Von Neumann regular:} an element $x$ of an evolution algebra $A$ is said {\it 
    von Neumann regular} if there is $y\in A$ satisfying $xyx=x$ (note that $(xy)x=x(yx)$ so we use the notation $xyx$). An evolution algebra $A$ is said to be {\it von Neumann regular} if every element is von Neumann regular.
    \item {\it Annihilator: } if  $A$ is an algebra then we define the {\it annihilator of} $A$ as $\ann(A):=\{x \colon x A = A x=0 \}$. We say that $A$ is a {\it zero annihilator algebra} if $\ann(A)=0$. In case $A$ is an evolution algebra, we have $\ann(A)=\Span{\{e_i \colon e_i^2=0\}}$. Similarly, if $I$ is an ideal of an algebra $A$ then  $\ann_A(I):=\{x\in A\colon xI=Ix=0\}$. 
    \item {\it Absorption property: } an ideal $I$ of a commutative $K$-algebra $A$ satisfies  \textit{the absorption property} if $x A\subset I$ implies $x\in I$. The \textit{absorption radical} of $A$ is the intersection of all ideals of A having the absorption property, denoted by $\rad(A)$.   
\end{enumerate}

    \begin{remark}\rm \label{primeabsorption}
        Notice that, in particular, prime ideals have the absorption property.
    \end{remark}

A \emph{directed graph} is a $4$-tuple,  $E=(E^0, E^1, r_E, s_E)$,  
consisting of two disjoint sets $E^0$ and $E^1$, along with two maps
$r_E, s_E: E^1 \to E^0$. We will use the terms "graph" and "directed graph" interchangeably.  The elements of $E^0$ are called \emph{vertices}, and the elements of 
$E^1$ are called \emph{edges} of $E$.  For $f\in E^1$, $r_E(f)$ and $s_E(f)$ are 
 the \emph{range} and the \emph{source} of $f$, respectively.  
If the graph we consider is clear from the context, we write $r(f)$ and $s(f)$. A
vertex $v$ for which $s^{-1}(v)=\emptyset$  is called a \emph{sink}, while a vertex $v$ for which $r^{-1}(v)=\emptyset$ is called a \emph{source}.  A \emph{path} $\mu$ of length $m$ is a finite chain of edges $\mu=f_1\ldots f_m$ such that $r(f_i)=s(f_{i+1})$ for $i=1,\ldots,m-1$.  We denote by $s(\mu):=s(f_1)$ the source of $\mu$ and $r(\mu):=r(f_m)$ the range of $\mu$.  If $s(\mu)=r(\mu)$ then we say that $\mu$ is a \emph{closed path based on $s(\mu)$}. A path is \emph{simple} if $s(e_i)\neq s(e_j)$ for all $i\neq j$. We write $\mu^0$, the set of vertices of $\mu$.  
The vertices will be considered trivial paths, namely, paths of length zero. A {\it cycle} is a closed simple path based at any of its vertices, and a cycle of length one is called a {\it loop}. We say that two vertices $u$ and $v$ satisfy
$u \geq v$ if there is a path from $u$ to $v$. If $A \subset E^0$, we write $u\geq A$ if $u \geq v$ for some $v\in A$.
We say that the graph $E$ is \emph{downward directed}: for every $u,v\in E^0$, there exists $z\in E^0$ such that $u\geq z$ and $v \geq z$.

A graph $E$ is said to be {\em row-finite}  if,  for any $u\in E^0$, we have 
that $\vert s^{-1}(u)\vert$ is finite. Additionally, it is said to satisfy \emph{Condition (Sing)}  if there is at most one edge between two vertices.  

It is worth recalling the of a graph $E$ modulo a subset $S\subset E^0$. This is a new graph $F=(F^0,F^1,r_F,s_F)$ such that $F^0:=E^0\setminus S$, $F^1:=\{f\in E^1\colon s(f),r(f)\notin S\}$, and $r_F,s_F$ are the suitable restrictions of $r_E, s_E$. The graph $F$ is usually denoted $E/ S$.

Furthermore, a directed graph is \textit{strongly connected}  if, given two different vertices, there exists a path that connects the first vertex to the second. A subset $H$ of $E^{0}$ is called \textit{hereditary} if, whenever $v\in H$ and
$w\in E^{0}$ satisfy $v\geq w$, then $w\in H$.

Next, we recall the definition of  the $\top$ topology, which makes sense in any graph (see \cite{CMMSS}).   
  Let $E$ be a graph and define $c\colon\P(E^0)\to\P(E^0)$ as the map such that for any $S\subseteq E^0$ we have $c(S):=\{v\in E^0 \; \vert \; v\ge A\}$. This map defines a {\it Kuratowski closure operator}, that is, it satisfies: 
\begin{itemize}
\item[(i)] $c(\vacio)=\vacio$, 
\item[(ii)] $S\subseteq c(S)$, 
\item[(iii)] $c(S)=c(c(S))$, and 
\item[(iv)] $c(S\cup T)=c(S)\cup c(T)$, for any $S,T\subseteq E^0$. 
\end{itemize}
Every Kuratowski closure operator induces 
a topology in $E^0$ whose closed sets are those $S\subseteq E^0$ such that $S=c(S)$ (see \cite[Chapter III, Section 5, Theorem 5.1]{DU}).   The topology induced by the $c$ described above is called the $\top$ topology.
So, an open set $O$ is one for which there is a subset $S\subseteq E^0$ such that $O=\{v\in E^0\colon v\not\ge S\}$.

Firstly, for any vertex $u\in E^0$, we say that $u\overset{\tiny 0}{\equiv} u$. Now, for vertices $u,v\in E^0$  we use the notation 
$u\overset{\tiny 1}{\equiv} v$ for $1$-connection, that is, either there exists $f\in E^1$ such that $s(f)=u$ and $r(f)=v$ or $s(f)=v$ and $r(f)=u$. Also, we say that $u$ and $v$ are {\it connected} if and only if there is a collection of vertices $u_1,\ldots,u_n$ such that
$$u=u_1\overset{\tiny i_1}{\equiv}
u_2\overset{\tiny i_2}{\equiv}\cdots\overset{\tiny i_n}{\equiv}u_n=v$$
\noindent
where $i_j \in \{0,1\}$ for all $j \in \{1,\ldots,n\}$. In this case, we write $u\equiv v$. In the sequel,  we prove that the connected components of $E^0$ endowed with the $\top$ topology are the equivalence classes of $E^0$ for the equivalence relation $\equiv$.

\begin{proposition} Let $E$ be a graph. The connected components of $E^0$, endowed with the $\top$ topology, are the equivalence classes of $E^0$ for the equivalence  $\equiv$.    
\end{proposition}
\begin{proof}

Let $S$ be a connected component of the topological space $(E^0,\top)$. Take $u\in S$. We are going to prove that $S=[u]$, where $[u]:=\{z\in E^0\colon z\equiv w\}$. Firstly, for any $w\in E^0$, we have that $[w]$ is closed in the $\top$ topology since $c([w])=[w]$. Also,  $[w]$ is open because $c(E^0\setminus [w]) = E^0\setminus [w]$. Indeed, if $v \in c(E^0  \setminus [w])$ then there is a vertex $q \in E^0  \setminus [w]$ such that $v \geq q$, which implies that $v \in [q] \neq [w]$, and so $v \in E^0 \setminus [w]$.  Next, we show that $S\subseteq [u]$. For this, observe that $S=(S\cap [u] )\cup (S\cap [u]^c)$. Since $[u]$ is clopen, $S$ is connected, and $S\cap [u]\neq \emptyset$, we conclude that $S\cap [u]^c = \emptyset$ and hence $S\subseteq [u]$ as desired. Now, we prove that $[u]\subseteq S$. Let us consider that $[u] = C_1 \cup C_2$ with $c(C_i) = C_i$ and $C_1\neq C_2$. Take vertices $w \in C_1$ and $w' \in C_2$. Since $w,w' \in [u]$, we have that there is a collection of vertices in $[u]$ such that $w = v_0 \overset{\tiny 1}{\equiv} v_1 \overset{\tiny 1}{\equiv} v_2 \overset{\tiny 1}{\equiv} \cdots \overset{\tiny 1}{\equiv} v_{n-1}\overset{\tiny 1}{\equiv} v_n = w'$. In this sequence, necessarily there must be $v_i \in C_1$ and $v_{i+1} \in C_2$, with $v_i \geq v_{i+1}$ or $v_{i+1} \geq v_i$. Without loss of generality, suppose that $v_i \geq v_{i+1}$. Then, $v_i \in c(C_2) = C_2$ and $C_1 \cap C_2 \neq \emptyset$ and $[u]$ is connected in the $\top$ topology. In consequence $[u] = S$.

Finally, since topological connected components are equivalence classes, the proposition statement follows.

\end{proof}

For any graph $E$ and a field $K$, we can construct the $K$-algebra $C(E^0,K)$ of all continuous maps $f\colon E^0\to K$, where $E^0$ is endowed with the $\top$ topology and $K$ with the discrete one. A map $f\in C(E^0,K)$ is constant along a connected component $S$ of $E^0$, because $S=\sqcup_{\lambda\in f(S)} (f^{-1}(\lambda)\cap S)$ and each $f^{-1}(\lambda)\cap S$ is open. Thus, $\vert f(S)\vert=1$. 
\begin{remark}\label{isomap} \rm
   Denote by $\mathfrak{C}$ the set of connected components of the graph $E$. Notice that the algebra $C(E^0,K)$ is isomorphic to the $K$-algebra $K^\mathfrak{C}$ of maps $\mathfrak{C}\to K$, where  $K^{\mathfrak{C}}$ is endowed with pointwise operations. 
\end{remark}

We will see later that the centroid of an evolution algebra contains the algebra $C(E^0,K)$, see Lemma~\ref{atarazana}. If the algebra has zero annihilator, then the monomorphism of Lemma~\ref{atarazana} is an isomorphism, see Lemma~\ref{isomega}.

Finally, we associate a graph to an evolution algebra in a similar way to what is done in \cite{Elduque2}: if $A$ is an evolution algebra with natural basis $B=\{e_i\}_{i\in \Lambda}$ and $M_B=(\omega_{ij})$, then we denote by $E=(E^0, E^1, r_E, s_E)$ the \emph{directed graph associated to $A$ relative to $B$}, where $E^0=\{e_i\}_{i \in \Lambda}$ and 
we draw an edge from $e_i$ to $e_j$ if and only if $\w_{ji}\ne 0$. Observe that a directed graph associated to an evolution algebra is row-finite and satisfies the Condition (Sing).  Notice that the graph associated to an evolution algebra may change with a change of basis. In the two-dimensional case, this is dealt with by looking at "squares" see \cite{squares}. Generally, when we mention the graph associated to an evolution algebra, we assume that a basis is fixed.

For an ideal $I$, of an evolution $K$-algebra $A$ with natural basis $B=\{e_i\}_{i \in \Lambda}$ and associated directed graph $E$, denote by $H_I$ the subset of $E^0$ given as 
$H_I:=\{e_i\in E^0\colon e_i^2\in I\}$. For any hereditary subset  $H\subset E^0$, define the subspace $I_H:=\oplus_{e_i\in H} K e_i =\Span{(H)}$ where $\Span{(H)}$ is the  set of all linear combinations of the elements of $H$. Notice that $I_{\emptyset}=\{0\}$.

\section{Graphical characterization of algebraic properties}

 In this section, we focus on characterizing algebraic properties of an evolution algebra in terms of an associated graph whenever possible. Moreover, we give two algorithms to analyze if an evolution algebra is nondegenerate and/or semiprime. We begin with the following result regarding the annihilator and absolute zero divisors of an evolution algebra.

\begin{proposition}\label{ostra}
Let $A$ be an evolution algebra with associated graph $E$ relative to the natural basis $B=\{e_i\}$. The following statements hold.
\begin{enumerate}[\rm (i)]
    \item \label{otro} $A$ is a zero annihilator algebra if and only if $E$ is sinkless. 

\item\label{jacare}
If $A$ is a semiprime, then it is a zero annihilator algebra. 

\item\label{papaya}
If $E^0$ has a subset $S\subsetneq E^0$ such that the graph $E/ S$ has a sink, then $A$ has nonzero absolute zero divisors, and hence, it is degenerate. Consequently, if $A$ is nondegenerate, then $E$ is sinkless (take $S=\emptyset$), and $A$ is a zero annihilator algebra. 

\item\label{acara} If the graph $E$ contains a vertex that is not the basis of a loop, then $A$ is degenerate. \end{enumerate}    

\end{proposition}
\begin{proof}
The first item is clear. Let us prove (\ref{jacare}): A graph $E$ associated to $A$ does not have sinks since otherwise there is a vertex $e_i$ with $e_i^2=0$. Then, the ideal $(e_i)=K e_i$ satisfies $(e_i)^2=0$, a contradiction. Thus, $E$ is sinkless, so $A$ is a zero annihilator algebra.
For item (iii), take $e_i\in E^0\setminus S$ a sink and observe that $e_i^2$ is in the linear span of $S$. So, as $e_i \notin S$, we have that $e_i e_i ^2=0$, and $e_i$ is a nonzero absolute zero divisor.  For item (iv), let $e_i$ be a vertex that is not the basis of a loop and consider $S=E^0\setminus\{e_i\}$. Then, applying (iii), we get that $A$ is degenerate.
\end{proof}

Note that the converse of the Proposition \ref{ostra} (\ref{acara}) is not valid in general as we will show in the following example.

\begin{example}\label{ejemplonode}\rm
Consider the evolution algebra $A$  with natural basis $B=\{e_i\}_{i=1}^2$  and multiplication defined by $e_1^2=e_1+e_2$, $e_2^2=-(e_1+e_2)$. The algebra is degenerate since the element $e_1+e_2$ is an absolute zero divisor. The graph of the algebra relative to the basis $B$ is the complete graph of order $2$: each vertex connects to itself and the other. We depict the graph below.

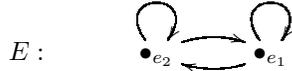
\begin{equation*} 
   E:   \xymatrix{
     & {\bullet}_{e_{2}}\ar@(ul,ur)\ar@/^/[r]  & {\bullet}_{e_{1}} \ar@(ul,ur) \ar@/^/[l]  &    \\
           }
           \end{equation*}

 We will prove that any graph associated to $A$ consists of the complete graph of order two, i.e., that the graph is invariant under a change of basis of the algebra. To see this, we 
first, make a general consideration: if $B$ is a natural basis of an evolution algebra, and we multiply each element in $B$ by a nonzero scalar, then we get another natural basis, say $B'$, and  the graphs associated to $B$ and $B'$ are isomorphic. 

Let $\{u_1,u_2\}$ be another natural basis of $A$ and write $u_1=\a e_1+\b e_2$, $v=\gamma e_1+\delta e_2$, where $\a,\b,\gamma,\delta\in K$. Then, we have that $\a \delta-\b \gamma\ne 0$ and $0=u_1v=(\a \gamma-\b \delta)e_1^2$, which gives $\a \gamma=\b \delta$ (since $e_1^2\ne 0$). So, we have that 
$$\begin{cases}\a \delta-\b \gamma\ne 0\cr \a \gamma-\b \delta=0.\end{cases}$$
In case $\gamma=0$, we have $\a ,\delta\ne 0$ hence $\b =0$. So, $u_1=\a e_1$ and $u_2=\delta e_2$. In this case, the new natural
basis has (up to isomorphism) the same graph as the old natural basis. Thus, assume $\gamma\ne 0$ and
put $\a =\b \delta/\gamma$. Then,  
$$\begin{cases}u_1=\frac{\b \delta}{\gamma}e_1+\b e_2\cr u_2=\gamma e_1+\delta e_2.\end{cases}$$
Notice that $\b\ne 0$, since if $\b =0$ then $u_1=0$ which is not possible. Therefore, the basis $\{u_1,u_2\}$ induces the same graph (always up to isomorphism) that $\{\frac{\gamma}{\b}u_1,u_2\}$. So, we can directly consider
the basis $$\begin{cases}u_1=\delta e_1+\gamma e_2\cr u_2=\gamma e_1+\delta e_2.\end{cases}$$
We highlight that $\delta^2\ne \gamma^2$ because $\{u_1,u_2\}$ is a basis. From here we get that 
$$e_1=\frac{1}{\delta^2-\gamma^2}(\delta u_1-\gamma u_2),\quad e_2=\frac{1}{\delta ^2-\gamma^2}(\delta u_2-\gamma u_1).$$
Finally, to find the graph of the algebra relative to $\{u_1,u_2\}$, we compute
$$u_1^2=(\delta-\gamma)(u_1+u_2)$$
$$u_2^2=(\gamma-\delta)(u_1+u_2)$$
and, since $\gamma\ne \delta$, the graph of the algebra relative to $\{u_1,u_2\}$ is also the complete graph of order $2$. 

\end{example}

    The following example proves that semiprimeness does not imply the absence of nonzero absolute zero divisors.
It also illustrates that having a zero annihilator does not imply nondegeneracy.

\begin{example}\rm Consider the evolution algebra $A$, over a field of characteristic other than two, with basis $B=\{e_i\}_{i=1}^4$  and multiplication $e_1^2 =   e_3+e_4, \
       e_2^2 =   -2e_3-2e_4, \
       e_3^2 =  e_1+e_2, \
       e_4^2 =  -e_1-e_2.$
The graph associated to $A$ is pictured below.
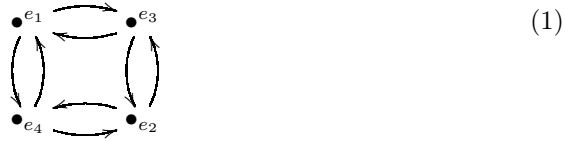
\begin{equation}
\xymatrix{
 \bullet^{e_1} \ar@/^/[r] \ar@/_/[d]
  &\bullet^{e_3} \ar@/_/[d] \ar@/^/[l] \\
{\bullet}_{e_4} \ar@/_/ [r] \ar@/_/[u] & \bullet_{e_2} \ar@/_/[u]\ar@/_/[l] }
\end{equation}
This algebra has zero annihilator because its graph is sinkless and is degenerate by Proposition~\ref{ostra} item~\eqref{papaya}. Furthermore, it is semiprime. Indeed, if $I$ is a nonzero ideal of $A$ then, by \cite[Lemma~3.2(2)]{conecting}, $H_I$ is a hereditary subset of $E^0$. Since the hereditary sets of this graph are $\emptyset$ and $E^0$, we have that either $H_I=\emptyset$ or $H_I=E^0$. In the first case, by \cite[Proposition~3.4(4)]{conecting}, we obtain that $I \subset I_{H_I} = I_\emptyset=0$, a contradiction. In the second case, $H_I=E^0$ implies that $A^2 \subset I$. In particular,  $ 0 \neq (e_3^2)^2=(e_1+e_2)^2 \in I^2 $ since $e_3^2\in A^2 \subset I$. Hence, $A$ is semiprime.

 \end{example}

There are algebraic properties, such as semiprimeness, that cannot be characterized in graphical terms. Indeed, consider the evolution algebras $A$ and $A'$, over a field of characteristic other than two, with natural basis $B=\{e_i\}_{i=1}^2$ and $B'=\{u_i\}_{i=1}^2$ respectively,  such that $e_1^2=e_1+e_2$, $e_2^2=-(e_1+e_2)$, $u_1^2=u_1+u_2$, $u_2^2=u_1-u_2$. Both have the same associated graph, but $A$ is not semiprime while $A'$ is. The same phenomenon happens with the perfection property: one of the algebras is perfect while the other is not. In conclusion:
\begin{proposition}\label{nocaracsemiprime}
 There is no graphical characterization of semiprimeness for evolution algebras. The same holds true for perfection.   
\end{proposition}

\begin{remark}\rm
Perfection does not imply nondegeneracy. For example, we can consider the evolution algebra with natural basis $B=\{e_1, e_2\}$ and multiplication given by $e_1^2 = e_2$ and $e_2^2 = e_1+e_2$. Observe that $e_1$ is an absolute zero divisor.
\end{remark}

In studying evolution algebras, it is usual to separate the perfect and non-perfect cases. Under the assumption of perfection of $A$, we can sharpen the conditions of nondegeneracy given in Proposition~\ref{ostra}. We do it in two ways: via graphs and via the structure matrix of the algebra.

\begin{proposition}\label{colorida}
    Let  $A$ be a perfect evolution algebra with natural basis $\{e_i\}_{i\in\Lambda}$ and let $E$ be the associated graph. Then, $A$ has a nonzero absolute zero divisor if and only if there is $S\subsetneq E^0$ such that the graph $E/ S$ consists of sinks. 
\end{proposition}

\begin{proof}
      If the set $S$ exists, then we have proved in Proposition~\ref{ostra}  that $A$ has nonzero absolute zero divisors. 
   Conversely, assume that there is $0\neq x\in A$ with $(xA)x=0$. Write $x= \sum_{t \in \Gamma} x_t e_t$, where $\Gamma := \sop(x)$ ($\Gamma^c := \Lambda \setminus\Gamma$). Let $i\in\Gamma$. Then, $(xe_i)x=0$ but $xe_i=x_ie_i^2=\sum_{q \in \Omega_i} x_i\w_{qi}e_q$, where $\Omega_i:=\sop(e_i^2)$. Hence 
   $$(xe_i)x = \sum_{q \in \Omega_i} x_i \w_{qi}e_q \sum_{t \in\Gamma} x_t e_t = \sum_{q \in \Omega_i\cap \Gamma} x_ix_q \w_{qi} e_q^2.$$
If $\Omega_i\cap \Gamma \neq \emptyset$ then, since perfection of $A$ implies that the set $\{e_q^2\}_{q\in \Omega_i\cap \Gamma}$ is linearly independent, we obtain that $x_ix_q\w_{qi}=0$ for any $i \in \Gamma $ and $q\in \Omega_i\cap \Gamma$. This implies that $\omega_{qi}=0$ for any $i \in \Gamma $ and $q\in \Omega_i\cap \Gamma$, because  $x_ix_q \neq 0$ for any $i \in \Gamma $ and $q\in \Omega_i\cap \Gamma$. Therefore, for $i \in \Gamma$, we have that
$$e_i^2= \sum_{q \in \Omega_i} \w_{qi}e_q= \sum_{q \in \Omega_i\cap \Gamma } \w_{qi}e_q + \sum_{q \in \Omega_i\cap \Gamma^c}  \w_{qi}e_q =  \sum_{q \in \Omega_i\cap \Gamma^c}  \w_{qi}e_q.$$ 
Observe that the equality above implies that  $\Omega_i = 
\sop(e_i^2) \subseteq \sop(x)^c = \Gamma^c$ for every $i \in \Gamma$. Define $S=\{e_t\in E^0 \colon t \in \Gamma^c\}$ and note that $S\ne E^0$, because $x\ne 0$. For the graph $E/ S$, we have: if $e_i\in E^0\setminus S$, then $e_i$ does not connect to any vertex in $E^0\setminus S$, and hence this graph consists of sinks.
\end{proof}

\subsection{Degeneracy and semiprimeness via matrices}\label{jabutiaqui} In terms of matrices, we provide a characterization of the degeneracy of perfect evolution algebras. Furthermore, we include two algorithms to verify if an evolution algebra is nondegenerate and/or semiprime.

\begin{proposition}
    Let $A$ be a perfect evolution algebra with natural basis $B=\{e_i\}_{i \in \Lambda}$ and structure matrix $M_B$. Then, $A$ is nondegenerate if and only if no submatrix $(\omega_{ij})_{i,j\in\Phi}$ of $M_B=(\omega_{ij})_{i,j\in\Lambda}$ is zero.
\end{proposition}
\begin{proof}  
Let  $0\ne x\in A$ be such that $(xA)x=0$. Equivalently, $(xe_i)x=0$ for every $i\in \Lambda$. Write  $x=\sum_{t\in \Gamma} x_t e_t$, with $\Gamma := \sop(x)$. Considering  $i\in \Gamma$, let $\Omega_i:=\sop(e_i^2)$. We have that $xe_i=\sum_{t\in \Gamma} x_te_t e_i=x_ie_i^2=\sum_{r \in \Omega_i} x_i\omega_{ri} e_r$  and $0=(xe_i)x=\sum_{r \in \Omega_i,j\in \Gamma} x_i\omega_{ri} e_r x_j e_j=\sum_{j\in \Omega_i\cap \Gamma}
x_ix_j\omega_{ji} e_j^2$. If $\Phi_i:=\Omega_i\cap \Gamma=\emptyset$ then $\omega_{ji}=0$ for every  $j \in \Gamma$. Otherwise, since $\{e_j^2\}_{j\in \Omega_i\cap \Gamma}$ is a linearly independent set, we conclude that $x_ix_j\omega_{ji}=0$ for every  $j \in \Phi_i$ and consequently $\w_{ji}=0$ for every $j \in \Phi_i$ and $i \in \Gamma$.

Conversely, if $(\omega_{ij})_{i,j\in\Phi}=0$ for certain $\emptyset \neq \Phi \subset \Lambda$, then any $x=\sum_{t\in\Phi}x_t e_t$ 
 is an absolute zero divisor (which can be chosen nonzero).  Indeed, for any $i\in\Phi$ and $\Omega_i:=\sop(e_i^2)$,  we have 
 $$(xe_i)x=(x_ie_i^2)\sum_{t\in\Phi} x_t e_t=
 \sum_{t\in \Phi \cap \Omega_i}x_ix_t\w_{ti}e_t^2=0.$$
 
\end{proof}

We can give a characterization of nondegeneracy in terms of the structure matrix $(\w_{ij})_{i,j=1}^n$ of a finite-dimensional evolution algebra: we introduce auxiliary variables $x_1,\ldots,x_n$ in the columns of the structure matrix as follows:
\begin{equation}\label{elex}
    N:=\begin{pmatrix}
    x_1\w_{11} & x_2\w_{12} & \cdots & x_i\w_{1i} & \cdots\\
    x_1\w_{21} & x_2\w_{22} & \cdots & x_i\w_{2i} & \cdots\\
    \vdots & \vdots & \vdots & \vdots & \vdots\\
    x_1\w_{n1} & x_2\w_{n2} & \cdots & x_i\w_{ni} & \cdots\\
\end{pmatrix}
\end{equation}
Then we collect all the polynomials $p_j=p_j(x_1,\ldots,x_n)$ given by the entries of $N^2$. If the system of all $p_j=0$ has a nonzero solution the algebra is degenerate and each nonzero solution $(a_1,\ldots,a_n)$ of the system provides the coordinates of a nonzero absolute zero divisor $\sum_i a_ie_i$. For instance,
consider the evolution $\mathbb{Q}$-algebra whose matrix relative to a natural basis $\{e_1,\ldots,e_4\}$ is 
$$\tiny\left(
\begin{array}{cccc}
 1 & 0 & 1 & -1 \\
 0 & 1 & 1 & 1 \\
 1 & 1 & 2 & 0 \\
 1 & 1 & 2 & 0 \\
\end{array}
\right).$$
Then $$\tiny N=\left(
\begin{array}{cccc}
 x & 0 & z & -t \\
 0 & y & z & t \\
 x & y & 2 z & 0 \\
 x & y & 2 z & 0 \\
\end{array}
\right),\ N^2=\left(
\begin{array}{cccc}
 -t x+x^2+x z & y z-t y & -2 t z+x z+2 z^2 & -t x \\
 t x+x z & t y+y^2+y z & 2 t z+y z+2 z^2 & t y \\
 x^2+2 x z & y^2+2 y z & x z+y z+4 z^2 & t y-t x \\
 x^2+2 x z & y^2+2 y z & x z+y z+4 z^2 & t y-t x \\
\end{array}
\right),
$$
and equating $N^2=0$ we get $x=y=z=0$ (the case $t=0$ implies the trivial solution) so that there is a nonzero solution $(0,0,0,1)$. We conclude that the algebra is degenerate and furthermore, $e_4$ is an absolute zero divisor. The algorithm is as follows:
for a commutative algebra $A$ the absolute zero divisors are those elements $x\in A$ such that $L_x^2=0$, but the matrix of $L_x$ for an $n$-dimensional evolution algebra with structure matrix $(\w_{ij})$ is \eqref{elex} (being 
$x=\sum_i x_i e_i$). Whence the absolute zero divisors of $0$ are given by the solutions of $N^2=0$.

Next, we provide below an algorithm to decide whether an evolution algebra is semiprime or not.

For a general algebra $A$, the ideals $I\triangleleft A$ with $I^2=0$ are contained in the set $\azd =\azd(A)$ of absolute zero divisors of $A$.
But $\azd$ is an algebraic variety (of degree two) and any $I$ is an algebraic
variety (actually a linear variety). Thus, if $I\subset \azd$, then $I$ is an irreducible component of $\azd(A)$. For instance, consider the evolution $\mathbb{Q}$-algebra with structure matrix
$$\tiny\left(
\begin{array}{cccc}
 1 & \llap{$-$}1 & 0 & 1 \\
 \llap{$-$}1 & 1 & 1 & 2 \\
 0 & 0 & 1 & 0 \\
 0 & 0 & 0 & 1 \\
\end{array}
\right),$$
relative to a natural basis $\{e_i\}$.
We want to compute all the ideals $I$ with $I^2=0$. 
$$\tiny N=\left(
\begin{array}{cccc}
 x & \llap{$-$}y & 0 & t \\
 \llap{$-$}x & y & z & 2 t \\
 0 & 0 & z & 0 \\
 0 & 0 & 0 & t \\
\end{array}
\right),\quad N^2=\left(
\begin{array}{cccc}
 x^2+x y & -x y-y^2 & -y z & t^2+t x-2 t y \\
 \llap{$-$}(x^2+x) y & x y+y^2 & y z+z^2 & 2 t^2-t x+2 t y \\
 0 & 0 & z^2 & 0 \\
 0 & 0 & 0 & t^2 \\
\end{array}
\right)$$
Then $\azd$ is given by the system $z=t=0=x(x+y)=y(x+y)$ and so the unique nontrivial subspace of $\azd$ is the given by $y=-x, z=t=0$, that is, $\mathbb{Q}(e_1-e_2)$. This is indeed an ideal of $A$ and the unique zero square ideal.

Instead of using the variety of absolute zero divisors $\azd(A)$, one can use
the algebraic variety of index-2 nilpotent elements $N_2(A):=\{a\in A\colon a^2=0\}$. Also, the zero square ideals are contained in $N_2(A)$ and we can proceed analogously as before, replacing $\azd(A)$ with $N_2(A)$. As a final example of this kind of analysis consider the evolution algebra of matrix
$$\tiny\left(
\begin{array}{ccccc}
 1 & \llap{$-$}1 & 0 & 0 & 1 \\
 1 & \llap{$-$}1 & 0 & 0 & 1 \\
 0 & 0 & 1 & \llap{$-$}1 & 1 \\
 0 & 0 & 0 & 0 & 1 \\
 0 & 0 & 0 & 0 & 1 \\
\end{array}
\right).$$ Directly from the matrix above, we can observe the existence of certain ideals: 
$\Q e_1+\Q e_2$ and $\Q e_3+\Q e_4$. However, we prove that there is no nonzero ideal $I$ with $I^2=0$. We compute as before
$$\tiny N=\left(
\begin{array}{ccccc}
 x & \llap{$-$}y & 0 & 0 & w \\
 x & \llap{$-$}y & 0 & 0 & w \\
 0 & 0 & z & \llap{$-$}t & w \\
 0 & 0 & 0 & 0 & w \\
 0 & 0 & 0 & 0 & w \\
\end{array}
\right),\quad N^2=\left(
\begin{array}{ccccc}
 x^2-x y & y^2-x y & 0 & 0 & w^2+w x-w y \\
 x^2-x y & y^2-x y & 0 & 0 & w^2+w x-w y \\
 0 & 0 & z^2 & \llap{$-$}t z & -t w+w^2+w z \\
 0 & 0 & 0 & 0 & w^2 \\
 0 & 0 & 0 & 0 & w^2 \\
\end{array}
\right)$$
and equating $N^2=0$ we get $z=w=0$, $x(x-y)=y(y-x)$. Thus, the only solutions
are $z=w=x=y=0$ which is the subspace $\Q e_4$, and $x=y$, $z=w=0$ which is the subspace $\Q(e_1+e_2)+\Q e_4$. So, these are the unique subspaces within the set of zero absolute divisors. But none of these subspaces is an ideal. So this algebra is semiprime though degenerate.

\subsection{Primeness} According to Proposition~\ref{nocaracsemiprime}, the semiprimeness property can not be characterized graphically. However, the primeness property can. For this, we need the remark below and the lemmas following it.

\begin{remark} \label{idealuno} \rm
    Let $A$ be an evolution algebra with natural basis $B=\{e_i\}$. It is easy to check that the ideal of $A$ generated by $e_i^2$, which we denote by $(e_i^2)$, is equal to $\Span{(\{e_k^2 \colon e_i \geq e_k\} )}$.
\end{remark}

\begin{lemma}\label{dormido}
     Consider an evolution $K$-algebra $A$ with natural basis $B =\{e_i\}_{i \in \Lambda}$, $M=(\omega_{ij})$ the structure matrix relative to $B$, $E$ the corresponding associated graph, and $i,j\in \Lambda$. The following assertions are equivalent.
     \begin{enumerate} [\rm (i)]
        
               \item There is a path from $e_i$ to $e_j$.
        
         \item $j \in \sop(I)$, where $I=(e_i^2)$.
     \end{enumerate}   
\end{lemma}
\begin{proof}

 Suppose that there is a path from $e_i$ to $e_j$. This implies that there exists a subset $\{e_{i_1}, \ldots,  e_{i_{n}}\}$ of $B$, with $e_i=e_{i_1}$ and $e_j=e_{i_n}$,  such that $i_k\in \sop(e_{i_{k-1}}^2)$ for $k\in \{2,\ldots,n\}$. Note that $e_{i_k}^2 \in I=(e_i^2)$ for $k\in \{1,\ldots,n\}$. In particular, $i_n=j\in \sop(e_{i_{n-1}}^2)\subset \sop(I)$. 
For the other implication, suppose that $j\in \sop(I)$. Then, $j \in \sop(x)$ for certain $x\in I$. Using  Remark~\ref{idealuno}, write $x=\sum_{e_i\geq e_k}\lambda_k e_k^2$. Observe that if $j \in \sop(x)$ then there exists $e_{k_0}$, with $e_i \geq e_{k_0}$, such that $j \in \sop(e_{k_0}^2)$. Note that $e_{k_0}\geq e_j$.  So, $e_i \geq e_{k_0} \geq e_j$.

    \end{proof}

\begin{proposition} \label{caractdirect}
    Consider an evolution $K$-algebra $A$ with natural basis $B =\{e_i\}_{i \in \Lambda}$, and let $E$ be the corresponding associated graph. The following assertions are equivalent.
    \begin{enumerate}[\rm (i)]
        \item For every pair of nontrivial ideals $I,J$ of $A$, $\sop(I) \cap \sop(J) \neq \emptyset$.
        \item The graph $E$ is downward directed. 
    \end{enumerate}
\end{proposition}
\begin{proof}
 Let $e_i, e_j \in E^0$. Suppose that $e_i^2=e_j^2=0$. In this case, we consider $I=(e_i)=Ke_i$ and $J=(e_j)=Ke_j$. Then, $\sop(I)=\{i\}$ and  $\sop(J)=\{j\}$, and hence $e_i=e_j$. Next, we consider the case $e_i^2=0 $ and $e_j^2\neq 0$. Let  $J=(e_j^2)$ and $I=(e_i)$. Then, $i \in \sop(J)$ because $\sop(I) \cap \sop(J) \neq \emptyset$. By Lemma~\ref{dormido} we have that $e_j \geq e_i$. For the final case, assume $e_i^2 \neq 0$  and $e_j^2 \neq 0 $. Define the nontrivial ideals $I= (e_i^2)$ and $J=(e_j^2)$. By hypothesis, $\sop(I)\cap \sop (J) \neq \emptyset$. Take $k \in \sop(I)\cap \sop (J)$.  Then, by Lemma~\ref{dormido}, $e_i\geq e_k$ and $e_j\geq e_k$.
 
 Conversely, consider two nontrivial ideals $I,J$ and take $0 \ne x \in I$ and $0 \neq y \in J$. Let $i \in \sop(x)$ and $j \in \sop(y)$. By hypothesis, there is a vertex $e_k \in E^0$ such that $e_i \geq e_k$ and $e_j \geq e_k$. By Lemma~\ref{dormido},  we have that $k \in \sop(e_i^2) \subset \sop (I)$ and $k \in \sop(e_j^2) \subset \sop(J)$. Then $k \in \sop(I) \cap \sop(J)$.

\end{proof}

\begin{remark} \rm
    For a semiprime  algebra $A$, it is easy to check that  the conditions $IJ=0$ and $I\cap J=0$ (for any two ideals) are equivalent:
If $I\cap J=0$, then $IJ\subset I\cap J=0$. Conversely, if $IJ=0$, then $(I\cap J)^2=0$ and by semiprimeness $I\cap J=0$.
\end{remark}

\begin{lemma}\label{nuevo}
    Let $A$ be an evolution $K$-algebra with natural basis $B=\{e_i\}_{i \in \Lambda}$ and $I,J$ two ideals of $A$. Then, $\sop(I) \cap \sop(J) = \emptyset$ implies that $IJ=0$. Furthermore, the converse is true if $A$ is semiprime.
\end{lemma}
\begin{proof}
    The first part is clear because if $x=\sum_{j \in \sop(x)}x_je_j$, and $y=\sum_{i \in \sop(y)}y_ie_i$, then $xy =\sum_{j\in \Gamma} x_j y_j e_j^2$, where $ \Gamma:= \sop(x) \cap \sop(y) $.  For the second part, assume that  $A$ is semiprime and suppose that  $IJ=0$. Then, $I\cap J=0$. This  implies that $\sop(I) \cap \sop(J) = \emptyset$. Indeed, if there exists $k\in \sop(I) \cap \sop(J) $, then $e_k^2 \in I \cap J=0$, but this is not possible because $A$ is semiprime, and so it is a zero annihilator algebra (Proposition~\ref{ostra}).
\end{proof}

\begin{theorem}\label{despierto}
    Suppose that $A$ is an evolution algebra, and $E$ is an associated graph. If $A$ is prime, then $E$ is downward directed.
\end{theorem}
\begin{proof}
    Suppose that $E$ does not satisfy the downward directed condition. By Proposition~\ref{caractdirect}, there exist nontrivial ideals $I$  and $J$ such that $\sop(I)\cap \sop(J)=\emptyset$. Thus, by Lemma~\ref{nuevo}, we obtain that $IJ= 0$, which contradicts the primeness of $A$.
\end{proof}

Next, we will see that the above theorem's reciprocal is not generally true.

\begin{example} \rm
  Consider the Example~\ref{ejemplonode}. 
Clearly, the graph satisfies the downward directed condition. Nevertheless, $A$ is not semiprime. Indeed, if we take $I=\Span{(e_1+e_2)}$, then $I^2=0$.
\end{example}

We prove below that the converse of  Theorem~\ref{despierto} holds for semiprime evolution algebras. 

\begin{corollary}\label{cor_prime_eq_down}
For a semiprime evolution algebra $A$, the following are equivalent:
\begin{enumerate}[\rm (i)]
    \item $A$ is prime.
    \item Any graph $E$ associated to $A$ is downward directed.
\end{enumerate}
\end{corollary}
\begin{proof}
By Theorem \ref{despierto}, we have that (i) implies (ii). Assume now that any graph associated to $A$ is downward directed. By Proposition~\ref{caractdirect},  for any nontrivial ideals $I,J$ of $A$, we have that $\sop(I)\cap \sop(J) \neq \emptyset$. Therefore, by Lemma~\ref{nuevo}, we obtain that $IJ\neq 0$.
\end{proof}

\begin{remark}\rm 
A prime evolution algebra is not necessarily perfect, as we show in the following example: let $A$ be an evolution algebra with natural basis $B=\{e_1,e_2\}$ such that $e_1^2= e_1$ and $e_2^2=e_1$. The associated graph is
\bigskip
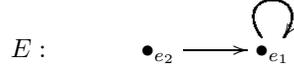
\begin{equation*} 
   E:   \xymatrix{
     & {\bullet}_{e_{2}} \ar[r]  & {\bullet}_{e_{1}} \ar@(ul,ur) &    \\
           }
           \end{equation*}

The only nontrivial ideal is $A^2 = \Span(e_1)$. Note that $A^2A^2 = A^2$. Consequently, the algebra is prime but not perfect.

Furthermore, as primeness entails semiprimeness, it follows that a semiprime evolution algebra is not necessarily perfect. The converse of this last statement is true, as we see in the following proposition.
\end{remark}

\begin{proposition}\label{prop_perfect_implies_semiprime}
    If $A$ is a perfect evolution algebra, then $A$ is semiprime.
\end{proposition}
\begin{proof}
 Let $B = \{e_i\}_{i\in \Lambda}$ be a natural basis for $A$. 
    If $I$ is a nontrivial  ideal of $A$, then there is an element $x \in I$ such as $\sop(x) \neq \emptyset$. Consider $e_i \in B$ with $i \in \sop(x)$. Then, $e_i^2 \in I$ and, since $A$ is perfect, we have that $e_i^2 \neq 0$. Moreover,  $(e_i^2)^2 \in I^2$ and $(e_i^2)^2 = \sum_{j\in \Lambda}\w_{ji}^2e_j^2$, which is non-zero since $A$ is perfect.
\end{proof}
 From the Proposition \ref{prop_perfect_implies_semiprime} and Corollary~\ref{cor_prime_eq_down}, we get the following graphical characterization of the primeness of perfect evolution algebras.

\begin{corollary}\label{cool}
A perfect evolution algebra $A$ is prime if and only if any graph $E$ associated to $A$ is downward directed.
\end{corollary}

\subsection{Prime ideals} In this subsection, we characterize the prime ideals of an evolution algebra $A$ with basis $B=\{e_i\}_{i \in \Lambda}$ and associated graph $E$. Recall that given an ideal $I$ the set $H_I:=\{e_i\in B:e_i^2\in I\}$ is an hereditary subset of $E^0$. 
Moreover, if $H$ is a hereditary subset of $E^0$, then the subspace $I_H=\bigoplus_{e_i\in H} K e_i$ is a basic ideal of $A$, see \cite[Lemma~3.2]{conecting}.

\begin{proposition}\label{rp3barulhento}
  Let $A$ be an evolution $K$-algebra. If $I$ is a prime ideal of $A$, then $I=I_{H_I}$.
\end{proposition}
\begin{proof}
    Suppose that $I$ is a prime ideal. By \cite[Proposition~3.4]{conecting} we have $I\subseteq I_{H_I}$, so we need to prove the other inclusion. Let $x\in I_{H_I}$ and write $x=\sum_{e_i\in H_I} \lambda_i e_i.$  Then, $x e_j \in I$ for all $j$ and so $x A \subseteq I$. Hence $[x] A/I =0$, where $ A/I$ is the quotient set modulo $I$ and $[x]$ is the equivalence class of $x$ in $A/I$. Therefore, since $I$ is prime, we obtain that $[x]=0$, i.e., $x\in I$.

\end{proof}

The converse of the proposition above is not true in general, as we show in the example below.

\begin{example}\label{ejemplonoprimo}\rm 
Consider the evolution algebra $A$ with natural basis $B=\{e_i\}_{i=1}^5$ and structure matrix 
$M_B = {\tiny\begin{pmatrix}
1 & 0 & 1 & 0 & 0 \\
0 & 1 & 1 & 0 & 0 \\
0 & 0 & 1 & 0 & 0 \\
0 & 0 & 1 & 1 & \llap{$-$}1 \\
0 & 0 & 0 & 1 & \llap{$-$}1 \\
\end{pmatrix}}.$  The associated graph is

\vspace{0.5cm}
\begin{equation*} 
   E:   \xymatrix{
      {\bullet}_{e_{2}} \ar@(ul,ur)  & {\bullet}_{e_{1}}  \ar@(ul,ur) & {\bullet}_{e_{4}} \ar@(ul,ur)  \ar@/^/[r] & {\bullet}_{e_{5}}\ar@(ul,ur)   \ar@/^/[l] \\
   & & {\bullet}_{e_{3}} \ar@(dl,dr) \ar[ul] \ar[ull] \ar[u]  &  \\
               }
\vspace{0.5cm}           \end{equation*}
Let $I= K e_4 \oplus K e_5$. In this case, $H_I=\{e_4,e_5\}$ and hence $I=I_{H_I}$. Nevertheless, $I$ is not a prime ideal. Indeed, notice that the graph of $A/I$ is obtained by removing $e_4$, $e_5$ and the arrows concomitant with those from the graph above and so it is not downward directed. Hence, by  Theorem \ref{despierto}, $A/I$ is not a prime algebra. 
\end{example}

Not every hereditary subset $H$ provides a prime ideal $I_H$, as the example above shows. However, we have the following characterization.
\begin{theorem}\label{charprimo}
    Let $A$ be an evolution algebra with $B=\{e_i\}_{i \in \Lambda}$ a natural basis and $I$ an ideal of $A$. Let $E$ be the associated graph to $A$ relative to $B$. Then the following assertions are equivalent:
    \begin{enumerate}[\rm (i)]
        \item $I$ is a prime ideal.
        \item $E/H_I$   is downward directed and $A/I$ is semiprime.
    \end{enumerate}
    \begin{proof}
        We begin with the proof of the implication  (i) $\Rightarrow$ (ii). Taking into account Remark~\ref{primeabsorption}, the ideal $I$ has the absorption property. Then,  \cite[Theorem 5.5]{conecting} implies that $I=I_{H_I}$. Since $I$ is prime, we obtain that $A/I$ is prime (hence semiprime) and by Corollary~\ref{cor_prime_eq_down} its associated graph is downward directed. Therefore, applying \cite[Theorem 6.5]{conecting}, we conclude that $E/H_I$ is downward directed. 
        To prove the converse, since $A/I$ is semiprime and $E/H_I$ is downward directed, again using Corollary~\ref{cor_prime_eq_down} and \cite[Theorem 6.5]{conecting}, we have that $A/I$ is prime and so $I$ is prime.      
    \end{proof}
\end{theorem}

The characterization in Theorem \ref{charprimo} is useful for any evolution algebra. For instance, in order to determine the prime ideals of the algebra in Example \ref{ejemplonoprimo}, we select all the hereditary subsets $H$ of $B$ such that $E/H$ is downward directed. These are 
$$H_1=\{e_1,e_2\}, \ H_2=\{e_1,e_4,e_5\}, \ H_3=\{e_2,e_4,e_5\}  \text{ and } H_4=\{e_1,e_2,e_4, e_5\}.$$
Now, we select from these all those such that $A/I_{H_i}$ is semiprime and these will produce the hereditary subsets $H$ such that $I_H$ is prime. Note that $A/I_{H_1}$ is not a semiprime algebra because the ideal generated by the element $[e_4+e_5]$  verifies that its square is zero. On the other hand, the other three hereditary subsets that remain produce semiprime algebras as $A/I_{H_i}$ $(i=2,3,4)$  are perfect evolution algebras (Proposition \ref{prop_perfect_implies_semiprime}). So, the only prime ideals are $Ke_1\oplus Ke_4 \oplus Ke_5$, $Ke_2\oplus Ke_4\oplus Ke_5$ and $Ke_1\oplus Ke_2\oplus Ke_4\oplus Ke_5$.

\subsection{Absorption property} 
In this section, we study the absorption property for an evolution algebra of any dimension. In the finite-dimensional case, our results recover some of the results of \cite{cadavid2023}. In fact, the proofs in \cite{cadavid2023} do not seem to require the finite dimensionality hypothesis and our statements are equivalent to theirs.

We begin this section by providing a graphical characterization of ideals of an evolution algebra having the absorption property.

Before we proceed, notice that it is easy to check that for any ideal $I$ of a commutative algebra $A$, the following assertions are equivalent:
    \begin{enumerate} [\rm (i)]
        \item The ideal $I$ has the absorption property.
        \item  $\ann(A/I)=0$.
        
    \end{enumerate}

\begin{proposition} Let $A$ be an evolution algebra with associated graph $E$, and $I=I_H$ for some hereditary subset $H$. Then the following are equivalent:
\begin{enumerate} [\rm (i)]
\item $I$ has the absorption property.
\item $E/H$ is sinkless.
\end{enumerate}
\end{proposition}

\begin{proof}
     We know that $I$ has the absorption property  if and only if $\ann(A/I)=0$. Applying \cite[Theorem 6.5 item (3)]{conecting}, one graph associated to $A/I$ is $E/H$. Now, by Proposition~\ref{ostra} \eqref{otro}, we get that  $\ann(A/I)=0$ if and only if $E/H$ is sinkless. 
\end{proof}

 The absorption radical of an evolution algebra can be described in terms of the sinks in various layers of a graph associated to the algebra. We make this precise below.

Following \cite{chains}, recall that if $A$ is a $K$-algebra then  $\ann^{(1)}(A):=\ann(A)$ and $\ann^{(i)}(A)$ is defined, by recurrence, as $\ann(A/\ann^{(i)}(A))=\ann^{(i+1)}(A)/\ann^{(i)}(A)$ for $i\in \mathbb{N}^{\ast}$. In this way, we have the following monotone increasing chain of ideals: $$\ann^{(1)}(A) \subset \ann^{(2)}(A) \ldots.$$ It is called the \textit{upper annihilating series.} Observe that $$\ann^{(i)}(A)=\{x \colon (\ldots((x\underbrace{A)A)\ldots )A}_{i}=0\}.$$ If there exists $k\in \mathbb{N}^{\ast}$ such that $k=\text{min}\{ q \colon \ann^{(q)}(A)=\ann^{(q+1)}(A)\}$, then we call it the \emph{annihilator stabilizing index} of $A$, denoted by $\asi(A)$. Now, define $\ann^{(\infty)}(A):= \cup_{n \geq 1} \ann^{(n)}(A)$. Taking into account \cite[Proposition~3.7 item (iii)]{chains}, we obtain that $\ann^{(\infty)}(A)=\rad(A)$ if $\asi(A) < \infty$.

Given a graph $E$, define a \textit{$1$-sink} of $E$ as a sink. Then, for $n>1$ define recursively an \textit{$n$-sink} as a vertex $u\in E^0$ such that $u$ is a sink in the graph obtained removing all $i$-sinks with $i<n$ (and edges concomitant with those). We denote the set of all $n$-sinks by $\mathbb S_n$. 

In the following graph $E$, we have $\mathbb S_1=\{e_7,e_8\}$ and 
$E_1:=E/\mathbb{S}_1$.
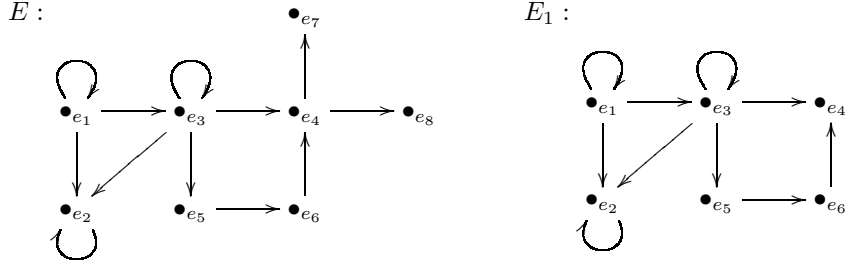
\begin{center}
    $$  {E:   \xymatrix{ & & {\bullet}_{e_{7}} &  & \\      {\bullet}_{e_{1}}  \ar@(ul,ur) \ar[d] \ar[r] & {\bullet}_{e_{3}} \ar@(ul,ur)  \ar[dl]\ar[r]\ar[d]  & {\bullet}_{e_{4}}\ar[u] 
\ar[r] &{\bullet}_{e_{8}}
  & \\
         {\bullet}_{e_{2}}  \ar@(dr,dl)  &{\bullet}_{e_{5}}\ar[r] & {\bullet}_{e_{6}}\ar[u] &  & }}{ E_1:   \xymatrix{ & & &  & \\     {\bullet}_{e_{1}}  \ar@(ul,ur) \ar[d] \ar[r] & {\bullet}_{e_{3}} \ar@(ul,ur)  \ar[dl]\ar[r]\ar[d]  & {\bullet}_{e_{4}} 
 &
  & \\
         {\bullet}_{e_{2}}  \ar@(dr,dl)  &{\bullet}_{e_{5}}\ar[r] & {\bullet}_{e_{6}}\ar[u] &  & }}
       $$\end{center}
       \vspace{1cm}
       
       Consequently $\mathbb{S}_2=\{e_4\}$. So
       $E_2:=E_1/\mathbb{S}_2$ is the graph below and $\mathbb{S}_3=\{e_6\}$.

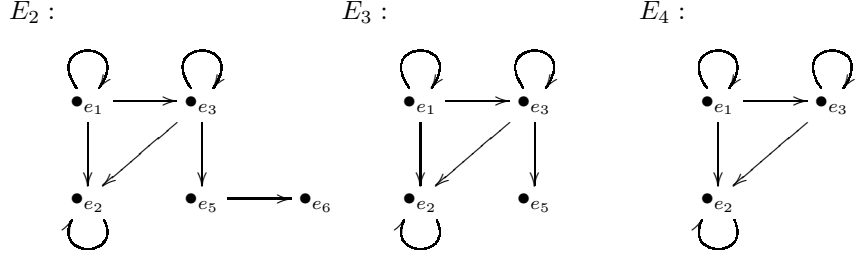
 $$
    {E_2:   \xymatrix{ & &    \\     {\bullet}_{e_{1}}  \ar@(ul,ur) \ar[d] \ar[r] & {\bullet}_{e_{3}} \ar@(ul,ur)  \ar[dl] \ar[d]  &  
 \\
         {\bullet}_{e_{2}}  \ar@(dr,dl)  &{\bullet}_{e_{5}}\ar[r] & {\bullet}_{e_{6}}    }}{E_3:   \xymatrix{ & &    \\     {\bullet}_{e_{1}}  \ar@(ul,ur) \ar[d] \ar[r] & {\bullet}_{e_{3}} \ar@(ul,ur)  \ar[dl] \ar[d]  &  \\
         {\bullet}_{e_{2}}  \ar@(dr,dl)  &{\bullet}_{e_{5}} &     }{E_4:   \xymatrix{ & &    \\     {\bullet}_{e_{1}}  \ar@(ul,ur) \ar[d] \ar[r] & {\bullet}_{e_{3}} \ar@(ul,ur)  \ar[dl]   & 
 \\
         {\bullet}_{e_{2}}  \ar@(dr,dl)  & &    } }}
       $$

\vspace{1cm}
Thus $\mathbb{S}_4=\{e_5\}$ and the graph $E_4$ is sinkless, so there is no more $n$-sinks.

With the setting above, we have the following graphical description of the absorption radical of an evolution algebra.

\begin{proposition}\label{paella}
    Let $A$ be an evolution $K$-algebra with associated graph $E$. If $\asi(A)< \infty$ then $\rad(A)=\Span(\bigcup_{n\geq 1}\{ \mathbb{S}_n \})$.
\end{proposition}
\begin{proof}
    This follows from $\rad(A)= \ann^{(\infty)}(A)$ and the fact that each $\ann^{(n)}(A)$ is generated by the $\mathbb{S}_n$.
\end{proof}

Notice that if $A$ is an evolution algebra with an associated graph equal to the graph $E$ given in the example above Proposition~\ref{paella}, then $\rad(A)=\Span \{e_4, e_5, e_6, e_7, e_8\}$.

\begin{remark} \rm
    The reader can check that the vertices that span the absorption radical of an evolution algebra, as in Proposition~\ref{paella}, are the ones that are not the basis of a closed path in the graph.
\end{remark}

\subsection{Von Neumann regularity} 

In this subsection, we study the von Neumann regularity property. We begin by searching necessary conditions for an element $x$ to be von Neumann regular in a finite evolution algebra $A$ with natural basis $B=\{e_i\}$ and structure matrix $M_B=(\omega_{ij})$.

Write $x = \sum_{i=1}^n x_ie_i$. Let us find a von Neumann inverse for $x$ in $A$. If we consider $y = \sum_{j = 1}^n y_je_j$, then $xy = \sum_{i=1}^n x_iy_ie_i^2 = \sum_{i,k = 1}^nx_iy_i\w_{ki}e_k$ and $xyx = \sum_{i,k=1}^nx_iy_i\w_{ki}x_ke_k^2 = \sum_{i,k,j=1}^n x_iy_i\w_{ki}x_k\w_{jk}e_j$. If $y$ is the von Neumann inverse of $x$, then we have that $xyx = x$, that is $\sum_{i,k=1}^n x_iy_i\w_{ki}x_k\w_{jk} = x_j$ for every $j = 1,\ldots,n$. We can rewrite these equations as
\begin{equation*}
    \sum_{i=1}^n \left ( \sum_{k=1}^n x_i\w_{ki}x_k\w_{jk}\right )y_i = x_j,
\end{equation*}
which is a linear system of equations depending on the parameters of $x$. We can write this system again in a different way. It is not difficult to check that $(M_B \cdot {\rm diag}(x))_{ij} = x_j\w_{ij}$. Then, 
\begin{equation*}
    \sum_{k=1}^n x_i\w_{ki}x_k\w_{jk} = \sum_{k=1}^n (M_B \cdot {\rm diag}(x))_{ki}(M_B\cdot {\rm diag}(x))_{jk} = (M_B \cdot {\rm diag}(x))^2_{ji},
\end{equation*}
where ${\rm diag}(x)=x \cdot {\rm Id}$ with ${\rm Id}$  being the identity matrix. Therefore, the coefficient matrix of the previous nonhomogeneous linear system is $(M_B\cdot {\rm diag}(x))^2$. From this, we obtain the following result.
\begin{proposition}
    If $A$ is a perfect finite dimensional evolution $K$-algebra, then all the elements outside the hyperplanes $x_i= 0$ are von Neumann regular. 
\end{proposition}

Next, we describe the natural basis elements of an evolution algebra, which are von Neumann regular.

\begin{lemma}
Let $A$ be an evolution $K$-algebra, and $B=\{e_i\}_{i\in \Lambda}$ a natural basis of $A$. An element $e_i$ of $B$ is von Neuman regular if and only if $e_i^2\in K^\times e_i$.    
\end{lemma}
\begin{proof}
    
If $e_i^2=\lambda e_i$ with $\lambda \in K^{\times}$, then $e_i (\frac{1}{\lambda^2} e_i) e_i = e_i$. For the converse, suppose that $e_i$ is regular, i.e., $e_ixe_i=e_i$ for some $x\in A$, and write $x=\sum_j \l_j e_j$. Then, 
$\l_i e_i^2 e_i=e_i$. Since $e_i^2=\sum_j\w_{ji}e_j$,  we obtain $\sum_j\l_i \w_{ji}e_j e_i=e_i$, that is, $\l_i\w_{ii}e_i^2=e_i$. Notice that $ \l_i\w_{ii}\neq 0$ since otherwise $e_i=0$.
 So, $\l_i =\w_{ii}^{-2}$ and $\w_{ji}=0$ if $j\ne i$. Hence $e_i^2=\w_{ii}e_i.$ 
 \end{proof}

 From the above paragraph, we obtain that if $A$ is von Neumann regular and $B=\{e_i\}$ a natural basis, then each element $e_i\in B$  satisfies $e_i^2\in K^\times e_i$ and hence the algebra is perfect and moreover associative, since $e_i^2e_j=e_i(e_ie_j)=0$ for $i\ne j$ (see \cite[Lema 3.2.1]{power}). Furthermore, each $Ke_i$ is a simple ideal of $A$, and $A$  is a direct sum of copies of the field. Thus, we have the following combinatorial characterization of the von Neumann regularity property.

 \begin{corollary}
 An evolution algebra $A$ is von Neumann regular if and only if every associated graph consists of isolated loops.
 \end{corollary}
 \begin{proof}
If A is von Neumann regular, then the result follows by the paragraph above. Conversely, an evolution algebra $A$ whose graphs consist of isolated loops is associative and semisimple, and hence von Neumann regular (see \cite[Corollary~4.24]{Lam}).
 \end{proof}

\begin{remark}\rm 
Consider an evolution $K$-algebra $A$ with natural basis $B=\{e_i\}_{i\in\Lambda}$ and associated graph $E$. We can consider the linear span of the set $\reg(B):=\{e_i\in B\colon e_i \text{ is von Neumann regular}\}$. It turns out that this linear span is an ideal of $A$ and a direct summand since it is generated by a subset of a natural basis. If we denote it by $I$,  we have a direct sum decomposition $A=I\oplus U$, where $U$ is a subspace. 
 The ideal $I$ is an evolution algebra with a natural basis given by the set of regular elements of $B$. Furthermore, it is von Neumann regular, and its associated graph is a set of isolated loops. The graph $E$ associated to $A$ relative to $B$ also contains these loops,  but typically, it will include other vertices that can eventually connect to the isolated loops.

\end{remark}

\section{The centroid of an evolution algebra}
In this section, we study the centroid of an evolution algebra and its connection with any graph associated to the algebra. Our first result describes the centroid of zero annihilator algebras that can be split into direct sums of ideals. Before we state the result, we need to recall a few definitions. 
Given a $K$-algebra $A$, a \textit{centralizer} $T \colon A \to A$ is a morphism of $A$-bimodules. In other words, $T$ is a linear application such that for every $x,y\in A$, $T(xy) = xT(y) = T(x)y$. The set of all centralizers, denoted by $\C(A)$, is called the \textit{centroid of $A$}. It is an algebra equipped with the sum and composition of applications. If $A$ is a prime algebra, $\C(A)$ is an extension field of $K$.

\begin{proposition}\label{rocks}
 Let $A$ be a zero annihilator $K$-algebra, which splits as a direct sum of ideals $A=\bigoplus_{\a\in\Lambda} I_\a$. Then:
 \begin{enumerate}[\rm (i)]
 \item Each $I_\a$, considered as an algebra on its own, is zero annihilator and, for each $\a$, we have 
 $\displaystyle \ann_A(I_\a)=\bigoplus_{\b \in \Lambda, \ \b\ne\a}I_\b$.
 \item \label{it2} If $T  $ is an element of the centroid $\C(A)$ then each $I_\a$ is $T$-invariant.  
 \item The centroid $\C(A)$ is isomorphic to the product $\prod_\a \C(I_\a)$.
 \end{enumerate}
\end{proposition}
\begin{proof}
The first assertion is clear by the definition of direct sum.

For the second assertion notice that $T(\ann_A(I_\a))\subset\ann_A(I_\a)$ for any $\a$. So, if we take an arbitrary $x\in I_\a$ and write $T(x)=\sum_\gamma z_\gamma$ (where $z_\gamma\in I_\gamma$), then for any $\b\ne \a$ we have $I_\b T(x)=T(I_\b x)=0$
whence $I_\b z_\b=0$ (and similarly $z_\b I_\b=0$). So, $z_\b\in\ann_{I_\b}(I_\b)=0$. Consequently $T(x)=z_\a\in I_\a$. For the third assertion, notice that we can construct a map  $\C(A)\overset{\w}{\to}\prod_{\a \in \Lambda} \C(I_\a)$ such that $T\mapsto (T\vert_{I_\a})$, where $T\vert_{I_\a}$ denotes the restriction of $T$ to $I_\a$, which is well defined by item \eqref{it2}. This map is a homomorphism of algebras, and it can be checked as a monomorphism. To see that it is also surjective, take an arbitrary collection $T_\a\in\C(I_\a)$ and define $T\colon A\to A$ as the map such that $T(\sum_\a x_\a):=\sum_\a T_\a(x_\a)$. Then, it can be easily proved that $T\in\C(A)$ and $\w(T)=(T_\a)_\a$.
\end{proof}

Let $A$ be an evolution algebra with natural basis $B = \{e_i\}_{i\in \Lambda}$ and structure matrix $M=(\w_{ij})$. Notice that the conditions for a linear map $T\colon A\to A$ to be a centralizer reduce to $T(e_i)e_j =0$ and $T(e_i^2)=T(e_i)e_i$ for every basis element $e_i\neq e_j$. If $T \in \C(A)$, we can write  $T(e_i) = \sum_{j \in \Lambda}t_{ji}e_j$. Thus, $0 = T(e_ie_j) = e_iT(e_j) =e_i\sum_{k \in \Lambda} t_{kj}e_k = t_{ij}e_i^2 = \sum_{k\in \Lambda} t_{ij}\w_{ki}e_k,$ which implies that 
\begin{equation}\label{primped}
\w_{ki}t_{ij} = 0\ \text{ for every $k \in \Lambda$ and $i \neq j$}.
\end{equation} 
In addition, we also have that  $T(e_i^2) = T \left (\sum_{j \in \Lambda} \w_{ji} e_j \right) = \sum_{j,k \in \Lambda} t_{kj}\w_{ji}e_k$ and  $e_iT(e_i) = t_{ii}e_i^2 = \sum_{k\in \Lambda} \w_{ki}t_{ii}e_k$, which leads to 
\begin{equation}\label{ec1}
\w_{ki}t_{ii}-\sum_{j \in \Lambda} t_{kj}\w_{ji} = 0, \textrm{ for every } i,k \in \Lambda.
\end{equation}

 The system of equations in the variables $\{t_{ij}\}_{i,j\in \Lambda}$, obtained in \eqref{ec1} jointly with \eqref{primped}, defines a linear homogeneous system of equations with solutions other than the trivial one since the identity is always a centralizer. 

Now, we recall the definition of $C(E^0,K)$ given in the paragraph above Remark~\ref{isomap}, where $E=(E^0,E^1,r_E,s_E)$ is the graph associated to the evolution algebra $A$ relative to $B$. Consider $f\in C(E^0,K)$. We define a linear map $T\colon A\to A$ such that 
$T(e_i)=f(e_i)e_i$ for any $i \in \Lambda$.
So $T(e_i)e_j=(f(e_i)e_i)e_j=0$ when $i\ne j$, and  $T(e_i)e_i=f(e_i)e_i^2=\sum_j f(e_i)\omega_{ji}e_j=\sum_j f(e_j)\omega_{ji}e_j=\sum_j\omega_{ji}T(e_j)=T(e_i^2)$ (take into account that $f(e_i)=f(e_j)$ if $\omega_{ji}\ne 0$). Summarizing, we have a map $C(E^0,K)\to \C(A)$ such that $f\mapsto T$.
It is easy to check that this is a $K$-algebra homomorphism, and further it is a monomorphism since $T=0$ implies $f(e_i)=0$ for any $i$. We state this result below.

\begin{lemma}\label{atarazana} Let $A$ be an evolution algebra with natural basis $B=\{e_i\}_{i\in \Lambda}$.
There is a $K$-algebra monomorphism $\Omega \colon C(E^0,K)\to\C(A)$ such that $f\mapsto T$, where $T$ is defined by
$T(e_i)=f(e_i)e_i$ for any $i \in \Lambda$.
\end{lemma}

In the specific case of a zero annihilator evolution algebra, we have that for all $i\in \Lambda$ there exists $k\in \Lambda$ such that $\w_{ki} \neq 0$. Moreover,  if $T \in \C(A)$ with  $T(e_i) = \sum_{j \in \Lambda}t_{ji}e_j$ hence applying  \eqref{primped}, we get $t_{ij}=0$ for all $j \neq i$ and for every $i \in \Lambda$. Therefore, $\w_{ki}t_{ii}-\sum_{j \in \Lambda} t_{kj}\w_{ji}=\w_{ki}(t_{ii}-t_{kk})=0$ for every  $i,k \in \Lambda$. Thus, for a zero annihilator
evolution algebra, the equivalent conditions for $T$ to be a centralizer simplify to:
\begin{equation}\label{sistcinc}
\begin{cases}
    t_{ij}=0 & \text{for any} \ i\ne j,\\
    \w_{ki}(t_{ii}-t_{kk})=0 & \text{for any}\ i,k.
\end{cases}
\end{equation}
So, in this case, the centralizers are diagonal maps relative to a natural basis, and any two of them commute. Furthermore, if $e_i$ and $e_j$ are in the same connected component of the graph relative to $B$, then $t_{ii} = t_{jj}$.   The centroid $\C(A)$ is therefore an associative, commutative subalgebra of $\hbox{End}_K(A)$ (the algebra of all $K$-linear maps $A\to A$). Moreover, we have the following characterization of the centroid.

\begin{lemma}
       \label{isomega}
    If $A$ is a zero annihilator evolution algebra, then the map $\Omega$ defined in Lemma~\ref{atarazana} is an isomorphism. Consequently, $\C(A)$ is isomorphic to $K^{\mathfrak{C}}$, where  $\mathfrak{C}$ is the set of connected components of any graph associated to $A$.
\end{lemma}
\begin{proof}
Let $B=\{e_i\}_{i\in \Lambda}$ be a natural basis of $A$ and $E=(E^0,E^1,r_E,s_E)$ the graph associated to $A$ relative to $B$. We prove that the map $\Omega$ of Lemma~\ref{atarazana} is surjective. 
Given $T\in\C(A)$ we have that $T(e_i)=t_{ii}e_i$ for each $i\in \Lambda$. Then, the map $f\colon E^0\to K$ such that
$f(e_i)=t_{ii}$ is continuous (being constant along connected components, see system \eqref{sistcinc}) and satisfies $\Omega(f)=T$. So, using Remark \ref{isomap}, we get that $\C(A)$ is isomorphic to $K^{\mathfrak{C}}$, where  $\mathfrak{C}$ is the set of connected components of the graph $E$. This implies that the set of connected components does not depend of the natural basis chosen.
\end{proof}

We have characterized the centralizer of zero annihilator evolution algebras. Next, we want to link the structure of the centroid with the structure of the algebra itself. Concretely, there is a one-to-one correspondence between simple summands of the centroid and indecomposable summands of the algebra. So, the structure of the centroid "reflects" in a way that of the algebra. To get this, we begin by developing some general concepts about the diagonalization of linear maps.

Let $V$ be a $K$-vector space and $\F$ any collection of linear maps $V\to V$. We say that $\F$ is {\em simultaneously diagonalizable} if there is a basis $B=\{e_i\}_{i\in\Lambda}$ of $V$ diagonalizing all the elements of $\F$, that is, for any $T\in\F$ and $i\in\Lambda$, one has $T(e_i)\in Ke_i$. Notice that such a family is commutative in the sense that $TS=ST$ for any $S,T\in\F$. 
We can write $T(e_i)=\l_i(T) e_i$, where $\l_i(T)\in K$ for any $i\in\Lambda$ and $T\in\F$. Furthermore, we define an equivalence relation on $\Lambda$ by declaring $i\sim j$ if and only if $\l_i(T)=\l_j(T)$ for any $T\in\F$. Denote $\Gamma:=\Lambda/\sim$, the quotient set modulo $\sim$, and by $[i]$ the equivalence class of $i\in\Lambda$, that is, $[i]=\{j\in\Lambda\colon \forall  \, T\in\F, \l_j(T)=\l_i(T)\}$.
Define $V_{[i]}:= \Span{(\{e_j\colon j\sim i\})}$, consequently,
\begin{equation}\label{simdiag}
V=\bigoplus_{\a\in\Gamma} V_\a
\end{equation}
Furthermore, it is easy to check that $V_{[i]} =\{x\in V\colon \forall \, T\in\F, T(x)=\l_i(T)x\}$. 

Now, apply this simultaneous diagonalization to the underlying vector space of the zero annihilator $K$-algebra $A$ and the family of simultaneously diagonalizable linear maps $\C(A)$ in order to obtain  a decomposition of $A $ as $$A=\bigoplus_{\a\in \Gamma} A_\a,$$
where for $\a=[i]$, one has $A_\a:=\{x\in A\colon\forall T\in\C(A), T(x)=\lambda_i(T)x\}$. We can say that $T\vert_{A_{[i]}}=\l_i(T)1_{A_{[i]}}$. Notice that each $A_\a$ is an ideal of $A$
and a zero annihilator algebra. Next, we characterize each $A_\alpha$.
\begin{proposition}
    In the setting above, each $A_\a$ is an indecomposable algebra whose centralizer is isomorphic to $K$.
\end{proposition}
\begin{proof}
    Let us prove first that $\C(A_\a)\cong K$ for any $\a$. Define the map $\C(A_\a)\overset{\theta}{\to} K$ by
    $T\mapsto \lambda_i(T)$, where $[i]=\a$. So, $\theta(T)$ is the scalar such that $T=\theta(T)1_{A_\a}$. Hence, for $S,T\in\C(A_\a)$ we have $ST=\theta(S)\theta(T)1_{A_\a}$ which implies $\theta(ST)=\theta(S)\theta(T)$. Clearly, $\theta$ is an injective $K$-algebra homomorphism. Therefore it is an isomorphism, since $\text{Im}(\theta)$ is not zero (for instance
    $\theta(1_{A_\a})=1$). Thus, $\C(A_\a)\cong K$ for any $\a$.
    
    Let us now prove the indecomposability of each $A_\a$. Assume that $A_\a=I\oplus J$, with $I,J\triangleleft A$ both nonzero.
    Consider the projections $\pi\colon A\to A$ such that $\pi\vert_I=1_I$ and $\pi(J)=0$ and similarly $\pi'\colon A\to A$ such that $\pi'\vert_J=1_J$ and $\pi'(I)=0$. Notice that both $\pi,\pi'\in\C(A)$ because taking elements $i_1,i_2\in I$ and $j_1,j_2\in J$ one has
    $\pi((i_1+j_1)(i_2+j_2))=i_1i_2=\pi(i_1+j_1)i_2=\pi(i_1+j_1)(i_2+j_2)$ and similarly the other identities required for $\pi\in\C(A)$. Also $\pi'=1_A-\pi$ is in $\C(A)$. Furthermore, the reader can check that $\{\pi,\pi'\}$ is a $K$-linearly independent set.  But this implies that $\dim(\C(A_\a))\ge 2$, which contradicts  $\C(A_\a)\cong K$. 
\end{proof}

Summarizing, we apply these results to evolution algebras and  we obtain the following theorem.

\begin{theorem}\label{praiajurere}
    For any zero annihilator evolution $K$-algebra $A$, we have a direct sum decomposition $A=\oplus_{\a\in\Gamma} A_\a$, where each $A_\a$ is an ideal that is indecomposable as algebra on its own, and $\C(A_\a)\cong K$. Moreover, $\C(A)\cong K^{\mathfrak{C}}$, where $\mathfrak{C}$ is the set of connected components of any graph associated to $A$.  Also, $\vert\Gamma\vert=\vert\mathfrak{C}\vert$.
\end{theorem}
\begin{remark}\rm
Notice that the above theorem implies that the graphs
corresponding to different natural basis have the same "number" of connected components for any two natural basis. More precisely, there is a one-to-one correspondence between the sets of components of both graphs. 
\end{remark}

It remains to be investigated if there is some kind of uniqueness in the previous decomposition. We can apply \cite[12.6 Theorem (Azumaya), p. 144]{anderson}. To do that, we consider the associative algebra  $\text{End}_K(A)$, that is, the algebra of all linear maps $A\to A$ with the composition product. Recall 
the unital multiplication algebra $\M(A)$, which is nothing but the subalgebra of $\text{End}_K(A)$ generated by $1_A$ and all the multiplication operators $L_x$, $x\in A$. Then, $\M(A)$ is an associative algebra, and $A$ is an $\M(A)$-module in a natural way. The $\M(A)$-submodules of $A$ are precisely the ideals of $A$, and indecomposable $\M(A)$-modules correspond with indecomposable ideals of $A$. Thus, the decomposition
$A=\oplus_{\a\in\Gamma}A_\a$ is a decomposition of the $\M(A)$-module $A$ as a direct sum of indecomposable modules. On the other hand, $\text{End}_{\M(A)}(A_\a)$ is just the centroid of $A_\a$, which is isomorphic to $K$, and hence it is a local ring, which enables us to apply \cite[12.6 Theorem (Azumaya), p. 144]{anderson}. We conclude that

\begin{theorem}\label{praiaMole}
Any two decompositions of a zero annihilator evolution algebra $A$ as in Theorem \ref{praiajurere} are unique up to reordering and isomorphism. 
\end{theorem}

Note that the zero annihilator hypothesis is not superfluous for the uniqueness of the above decomposition. Indeed, consider the isomorphic degenerate evolution algebras $A$ and $A'$,  with natural basis $B=\{e_i\}_{i=1}^2$ and $B'=\{u_i\}_{i=1}^2$ respectively, such that $e_1^2=e_1$, $e_2^2=0$, $u_1^2=0$ and $u_2^2=u_1+u_2$. They do not have the same number of connected components. We picture the graphs of the algebras below.

\bigskip

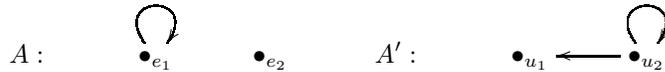
\begin{equation*} 
   A:   \xymatrix{
     & {\bullet}_{e_{1}}\ar@(ul,ur)& {\bullet}_{e_{2}}   &    \\
           }   
   A':   \xymatrix{
     & {\bullet}_{u_{1}}& {\bullet}_{u_{2}}\ar@(ul,ur) 
\ar[l] &    \\
           }
           \end{equation*}

As a consequence of our analysis, we compute the centroid of any prime evolution algebra.
           
\begin{corollary}
    If $A$ is a prime evolution algebra, then $\C(A)\cong K$.
\end{corollary}
\begin{proof}
If $A$ is prime, then, by Theorem~\ref{despierto}, the associated graph is downward directed and so has only one connected component. Furthermore,  by Proposition~\ref{ostra}(\ref{papaya}) $A$ is a zero annihilator algebra. So, applying Theorem~\ref{praiajurere}, we obtain the desired result.
\end{proof}

\section*{Acknowledgement}

Part of the research work that led us to this article was carried out during research stays in Brazil by the first, fourth, fifth, and sixth authors. These authors thank the Federal University of Santa Catarina for their hospitality and generosity. The research leading to this article was conducted, in part, during research stays in Spain by the second and third authors. These authors express their gratitude to the Universidad de Málaga for its warm hospitality and generous support. The third author was partially supported by Capes-Print Brazil, Conselho Nacional de Desenvolvimento Cient\'ifico e Tecnol\'ogico (CNPq) - Brazil, and Funda\c{c}\~ao de Amparo \`a Pesquisa e Inova\c{c}\~ao do Estado de Santa Catarina (FAPESC). The fourth and fifth authors were supported by the Brazilian Federal Agency for Support and Evaluation of Graduate Education â Capes with Process numbers: 88887.895676/2023-00 and 88887.895611/2023-00 respectively. The first,  fourth, fifth, and sixth authors are supported by the Spanish Ministerio de Ciencia e Innovaci\'on   through project  PID2019-104236GB-I00/AEI/10.13039/501100011033 and by the Junta de Andaluc\'{i}a  through projects  FQM-336 and UMA18-FEDERJA-119, all of them with FEDER. The sixth author is supported by a Junta de Andalucía PID fellowship no. PREDOC\_00029.  The Junta de Andalucía partially supports the second and third authors through project FQM-336.

\section{Declarations}

\subsection*{Ethical Approval:}

This declaration is not applicable.

\subsection*{Conflicts of interests/Competing interests:} We have no conflicts of interests/competing interests to disclose.

\subsection*{Authors' contributions:}

All authors contributed equally to this work. 

\subsection*{Data Availability Statement:} The authors confirm that the data supporting the findings of this study are available within the article.
\vskip 1cm

\bibliographystyle{acm}
\bibliography{ref}

\end{document}